\documentclass[reqno, 12pt]{amsart}
\makeatletter
\let\origsection=\section \def\section{\@ifstar{\origsection*}{\mysection}}
\def\mysection{\@startsection{section}{1}\z@{.7\linespacing\@plus\linespacing}{.5\linespacing}{\normalfont\scshape\centering\S}}
\makeatother

\usepackage[utf8]{inputenc}
\usepackage{amsmath,amssymb,amsthm}
\usepackage{mathrsfs}
\usepackage{dsfont}
\usepackage{mathabx}\changenotsign

\usepackage{microtype}

\usepackage{xcolor}
\usepackage[backref=page]{hyperref}
\hypersetup{    colorlinks,
    linkcolor={red!60!black},
    citecolor={green!60!black},
    urlcolor={blue!60!black}
}

\usepackage{bookmark}

\usepackage[abbrev,msc-links,backrefs]{amsrefs}
\usepackage{doi}

\renewcommand{\PrintDOI}[1]{\doi{#1}}

\usepackage[T1]{fontenc}
\usepackage{lmodern}

\usepackage[english]{babel}

\usepackage{tikz}

\usepackage{geometry}
\usepackage{enumitem}
\newcommand\rmlabel{\upshape({\itshape\roman*\,\/})}

\newcommand\alabel{\upshape({\itshape\alph*\,\/})}

\let\polishlcross=\l
\def\l{\ifmmode\ell\else\polishlcross\fi}

\newcommand\qand{\quad\text{and}\quad}

\renewcommand{\emptyset}{\varnothing}
\renewcommand{\setminus}{\smallsetminus}

\makeatletter
\def\moverlay{\mathpalette\mov@rlay}
\def\mov@rlay#1#2{\leavevmode\vtop{   \baselineskip\z@skip \lineskiplimit-\maxdimen
   \ialign{\hfil$\m@th#1##$\hfil\cr#2\crcr}}}
\newcommand{\charfusion}[3][\mathord]{
    #1{\ifx#1\mathop\vphantom{#2}\fi
        \mathpalette\mov@rlay{#2\cr#3}
      }
    \ifx#1\mathop\expandafter\displaylimits\fi}
\makeatother

\newcommand{\dcup}{\charfusion[\mathbin]{\cup}{\cdot}}

\newtheorem{theorem}{Theorem}
\newtheorem{lemma}[theorem]{Lemma}
\newtheorem{corollary}[theorem]{Corollary}

\newcommand\ccb{{\mathcal{B}}}
\newcommand\ccc{{\mathcal{C}}}

\newcommand\cch{{\mathcal{H}}}

\newcommand\ccp{{\mathcal{P}}}

\newcommand\ccq{{\mathcal{Q}}}

\newcommand\ccx{{\mathcal{X}}}
\newcommand\bbn{{\mathds{N}}}

\renewcommand{\phi}{\varphi}
\renewcommand{\rho}{\varrho}
\renewcommand{\theta}{\vartheta}
\newcommand{\eps}{\varepsilon}
\newcommand{\epsprime}{\eps'}
\newcommand{\Aeps}{A_{\eps}}
\newcommand{\Aast}{A_{\ast}}
\newcommand{\aeps}{a_{\eps}}
\newcommand{\Beps}{B_{\eps}}
\newcommand{\Bast}{B_{\ast}}
\newcommand{\beps}{b_{\eps}}
\newcommand{\Veps}{V_{\eps}}
\newcommand{\veps}{v_{\eps}}

\begin{document}

\title[Loose Hamiltonian cycles forced by $(k-2)$-degree]{Loose Hamiltonian cycles forced by large $(k-2)$-degree \\ -- sharp version --}

\author[J.~de~O.~Bastos]{Josefran de Oliveira Bastos}
\author[G.~O.~Mota]{Guilherme Oliveira Mota}
\address{Instituto de Matem\'atica e Estat\'{\i}stica, Universidade de
   S\~ao Paulo, S\~ao Paulo, Brazil}
\email{\{josefran|mota\}@ime.usp.br}

\author[M.~Schacht]{Mathias Schacht}
\author[J.~Schnitzer]{Jakob Schnitzer}
\author[F.~Schulenburg]{Fabian Schulenburg}
\address{Fachbereich Mathematik, Universit\"at Hamburg, Hamburg, Germany}
\email{schacht@math.uni-hamburg.de}
\email{\{jakob.schnitzer|fabian.schulenburg\}@uni-hamburg.de}

\thanks{The first author was supported by CAPES\@.
The second author was supported by FAPESP (Proc. 2013/11431-2 and 2013/20733-2) and CNPq (Proc. 477203/2012-4 and {456792/2014-7}).
The cooperation was supported by a joint CAPES/DAAD PROBRAL (Proc. 430/15).}

\begin{abstract}
We prove for all $k\geq 4$ and $1\leq\l<k/2$ the sharp minimum $(k-2)$-degree bound for a $k$-uniform hypergraph~$\cch$ on~$n$ vertices to contain a Hamiltonian $\l$-cycle if $k-\l$ divides~$n$ and~$n$ is sufficiently large.
This extends a result of Han and Zhao for $3$-uniform hypegraphs.
\end{abstract}

\keywords{hypergraphs, Hamiltonian cycles, degree conditions}
\subjclass[2010]{05C65 (primary), 05C45 (secondary)}

\maketitle

\section{Introduction}

Given $k\geq 2$, a $k$-uniform hypergraph $\cch$ is a pair $(V,E)$ with vertex set $V$ and edge set $E\subseteq V^{(k)}$, where $V^{(k)}$ denotes the set of all $k$-element subsets of $V$.
Given a $k$-uniform hypergraph $\cch=(V,E)$ and a subset $S \in V^{(s)}$, we denote by $d(S)$ the number of edges in $E$ containing~$S$ and we denote by $N(S)$ the $(k-s)$-element sets $T\in V^{(k-s)}$ such that $T\dcup S\in E$, so $d(S)=|N(S)|$.
The \emph{minimum $s$-degree} of $\cch$ is denoted by $\delta_s(\cch)$ and it is defined as the minimum of $d(S)$ over all sets $S\in V^{(s)}$.
We denote by the \textit{size} of a hypergraph the number of its edges.

We say that a $k$-uniform hypergraph $\ccc$ is an \emph{$\l$-cycle} if there exists a cyclic ordering of its vertices such that every edge of $\ccc$ is composed of $k$ consecutive vertices, two (vertex-wise) consecutive edges share exactly $\l$ vertices, and every vertex is contained in an edge.
Moreover, if the ordering is not cyclic, then $\ccc$ is an \emph{$\l$-path} and we say that the first and last~$\l$ vertices are the ends of the path.
The problem of finding minimum degree conditions that ensure the existence of Hamiltonian cycles, i.e.\ cycles that contain all vertices of a given hypergraph, has been extensively studied over the last years (see, e.g., the surveys~\cites{RRsurv,Zhao-survey}).
Katona and Kierstead~\cite{KaKi99} started the study of this problem, posing a conjecture that was confirmed by R\"odl, Ruci\'nski, and Szemer\'edi~\cites{RoRuSz06,RoRuSz08}, who proved the following result:
For every $k\geq 3$, if $\cch$ is a $k$-uniform $n$-vertex hypergraph with $\delta_{k-1}(\cch)\geq {(1/2+o(1))}n$, then $\cch$ contains a Hamiltonian $(k-1)$-cycle.
K\"uhn and Osthus proved that $3$-uniform hypergraphs~$\cch$ with $\delta_2(\cch)\geq {(1/4 +o(1))}n$ contain a Hamiltonian $1$-cycle~\cite{KuOs06}, and H\`an and Schacht~\cite{HaSc10} (see also~\cite{KeKuMyOs11}) generalized this result to arbitrary $k$ and $\l$-cycles with $1\leq \l <k/2$.
In~\cite{KuMyOs10}, K\"uhn, Mycroft, and Osthus generalized this result to $1\leq \l<k$, settling the problem of the existence of Hamiltonian $\l$-cycles in $k$-uniform hypergraphs with large minimum $(k-1)$-degree.
In Theorem~\ref{theorem:asymp} below (see~\cites{BuHaSc13,BaMoScScSc16+}) we have minimum $(k-2)$-degree conditions that ensure the existence of Hamiltonian $\l$-cycles for $1\leq \l<k/2$.

\begin{theorem}\label{theorem:asymp}
For all integers $k\geq 3$ and $1\leq \l<k/2$ and every $\gamma>0$ there exists an $n_0$ such that every $k$-uniform hypergraph $\cch=(V,E)$ on $|V|=n\geq n_0$ vertices with $n\in(k-\l)\bbn$ and
\begin{equation*}
    \delta_{k-2}(\cch)\geq\left(\frac{4(k-\l)-1}{4{(k-\l)}^2}+\gamma\right)\binom{n}{2}
\end{equation*}
contains a Hamiltonian $\l$-cycle.
\qed
\end{theorem}

The minimum degree condition in Theorem~\ref{theorem:asymp} is asymptotically optimal as the following well-known example confirms.
The construction of the example varies slightly depending on whether~$n$ is an odd or an even multiple of~$k-\l$.
We first consider the case that $n = (2m + 1)(k-\l)$ for some integer~$m$.
Let $\ccx_{k,\l}(n)=(V,E)$ be a $k$-uniform hypergraph on $n$ vertices such that an edge belongs to $E$ if and only if it contains at least one vertex from $A \subset V$, where $|A|=\left\lfloor \frac{n}{2(k-\l)} \right\rfloor$.
It is easy to see that $\ccx_{k,\l}(n)$ contains no Hamiltonian $\l$-cycle, as it would have to contain $\frac{n}{k-\l}$ edges and each vertex in~$A$ is contained in at most two of them.
Indeed any maximal $\l$-cycle includes all but $k-\l$ vertices and adding any additional edge to the hypergraph would imply a Hamiltonian $\l$-cycle.
Let us now consider the case that $n= 2m(k-\l)$ for some integer~$m$.
Similarly, let $\ccx_{k,\l}(n)=(V,E)$ be a $k$-uniform hypergraph on $n$ vertices that contains all edges incident to $A \subset V$, where $|A|=\frac{n}{2(k-\l)}-1$.
Additionally, fix some $\ell+1$ vertices of $B = V\setminus A$ and let $\ccx_{k,\l}(n)$ contain all edges on $B$ that contain all of these vertices, i.e., an $(\l+1)$-star.
Again, of the $\frac{n}{k-\l}$ edges that a Hamiltonian $\l$-cycle would have to contain, at most $\frac{n}{k-\l} - 2$ can be incident to $A$.
So two edges would have to be completely contained in $B$ and be disjoint or intersect in exactly $\l$ vertices, which is impossible since the induced subhypergraph on $B$ only contains an $(\l+1)$-star.
Note that for the minimum $(k-2)$-degree the $(\l+1)$-star on $B$ is only relevant if $\l=1$, in which case this star increases the minimum $(k-2)$-degree by one.

In~\cite{HaZh15b}, Han and Zhao proved the exact version of Theorem~\ref{theorem:asymp} when $k=3$, i.e., they obtained a sharp bound for $\delta_{1}(\cch)$.
We extend this result to $k$-uniform hypergraphs.

\begin{theorem}[Main Result]\label{theorem:main}
For all integers $k\geq 4$ and $1\leq \l<k/2$ there exists $n_0$ such that every $k$-uniform hypergraph $\cch=(V,E)$ on $|V|=n\geq n_0$ vertices with $n\in(k-\l)\bbn$ and
\begin{equation}\label{eq:sharp_minimum_degree}
    \delta_{k-2}(\cch)
    >
    \delta_{k-2}(\ccx_{k,\l}(n))
\end{equation}
contains a Hamiltonian $\l$-cycle.
In particular, if
\begin{equation*}
    \delta_{k-2}(\cch)
    \geq
    \frac{4(k-\l)-1}{4{(k-\l)}^2} \binom{n}{2},
\end{equation*}
then $\cch$ contains a Hamiltonian $\l$-cycle.
\end{theorem}

The following notion of extremality is motivated by the hypergraph $\ccx_{k,\l}(n)$.
A $k$-uniform hypergraph $\cch=(V,E)$ is called \emph{$(\l,\xi)$-extremal} if there exists a partition $V=A\dcup B$ such that $|A|=\left\lceil \frac{n}{2(k-\ell)} - 1 \right\rceil$, $|B|=\left\lfloor \frac{2(k-\l)-1}{2(k-\l)}n + 1 \right\rfloor$ and $e(B)=|E\cap B^{(k)}|\leq \xi \binom{n}{k}$.
We say that $A\dcup B$ is an \emph{$(\l,\xi)$-extremal partition} of $V$.
Theorem~\ref{theorem:main} follows easily from the next two results, the so-called \emph{extremal case} (see Theorem~\ref{theorem:extremal} below) and the \emph{non-extremal case} (see Theorem~\ref{theorem:non-extremal}).

\begin{theorem}[Non-extremal Case]\label{theorem:non-extremal}
For any $0<\xi<1$ and all integers $k\geq 4$ and $1\leq \l<k/2$, there exists $\gamma>0$ such that the following holds for sufficiently large $n$.
Suppose $\cch$ is a $k$-uniform hypergraph on $n$ vertices with $n\in(k-\l)\bbn$ such that $\cch$ is not  $(\l,\xi)$-extremal and
\begin{equation*}
    \delta_{k-2}(\cch)\geq\left(\frac{4(k-\l)-1}{4{(k-\l)}^2}-\gamma\right)\binom{n}{2}.
\end{equation*}
Then $\cch$ contains a Hamiltonian $\l$-cycle.
\qed
\end{theorem}
The non-extremal case was the main result of~\cite{BaMoScScSc16+}.
\begin{theorem}[Extremal Case]\label{theorem:extremal}
For any integers $k\geq 3$ and $1\leq \l<k/2$, there exists $\xi>0$ such that the following holds for sufficiently large $n$.
Suppose $\cch$ is a $k$-uniform hypergraph on $n$ vertices with $n\in(k-\l)\bbn$ such that $\cch$ is $(\l,\xi)$-extremal and
\begin{equation*}
    \delta_{k-2}(\cch)
    >
    \delta_{k-2}(\ccx_{k,\l}).
\end{equation*}
Then $\cch$ contains a Hamiltonian $\l$-cycle.
\end{theorem}

In Section~\ref{sec:overview} we give an overview of the proof of Theorem~\ref{theorem:extremal} and state Lemma~\ref{lem:mainlemma}, the main result required for the proof.
In Section~\ref{sec:mainproof} we first prove some auxiliary lemmas and then we prove Lemma~\ref{lem:mainlemma}.

\section{Overview}\label{sec:overview}

Let $\cch=(V,E)$ be a $k$-uniform hypergraph and let $X,Y\subset V$ be disjoint subsets.
Given a vertex set $L\subset V$ we denote by $d(L,X^{(i)} Y^{(j)})$ the number of edges of the form $L \cup I \cup J$, where $I \in X^{(i)}$, $J \in Y^{(j)}$, and $|L| + i + j = k$.
We allow for $Y^{(j)}$ to be omitted when $j$ is zero and write $d(v,X^{(i)} Y^{(j)})$ for $d(\{ v \},X^{(i)} Y^{(j)})$.

The proof of Theorem~\ref{theorem:extremal} follows ideas from~\cite{HaZh15}, where a corresponding result with a $(k-1)$-degree condition is proved.
Let $\cch=(V,E)$ be an extremal hypergraph satisfying~\eqref{eq:sharp_minimum_degree}.
We first construct an $\l$-path~$\ccq$ in $\cch$ (see~Lemma~\ref{lem:mainlemma} below) with ends $L_0$ and $L_1$ such that there is a partition $\Aast\dcup \Bast$ of $(V\setminus \ccq) \cup L_0 \cup L_1$ composed only of ``typical'' vertices (see~\ref{it:2-mainlemma} and~\ref{it:3-mainlemma} below).
The set $\Aast\cup \Bast$ is suitable for an application of Lemma~\ref{lem:3.10} below, which ensures the existence of an $\l$-path $\ccq'$ on $\Aast\cup \Bast$ with $L_0$ and $L_1$ as ends.
Note that the existence of a Hamiltonian $\l$-cycle in $\cch$ is guaranteed by $\ccq$ and $\ccq'$.
So, in order to prove Theorem~\ref{theorem:extremal}, we only need to prove the following lemma.

\begin{lemma}[Main Lemma]\label{lem:mainlemma}
    For any $\rho > 0$ and all integers $k\geq 3$ and $1\leq \l <k/2$, there exists a positive~$\xi$ such that the following holds for sufficiently large~$n\in (k-\l)\bbn$.
    Suppose that~$\cch=(V,E)$ is an $(\l,\xi)$-extremal $k$-uniform hypergraph on~$n$ vertices and
    \[
        \delta_{k-2}(\cch)
        >
        \delta_{k-2}(\ccx_{k,\l}(n)).
    \]
    Then there exists a non-empty $\l$-path $\ccq$ in $\cch$ with ends $L_0$ and $L_1$ and a partition $\Aast\dcup \Bast=(V\setminus \ccq) \cup L_0 \cup L_1$ where $L_0, L_1\subset \Bast$ such that the following hold:
    \begin{enumerate}[label=\rmlabel]
        \item\label{it:1-mainlemma} $|\Bast|=(2k-2\l-1)|\Aast|+\l$,
        \item\label{it:2-mainlemma} $d(v,\Bast^{(k-1)})\geq (1-\rho) \binom{|\Bast|}{k-1}$ for any vertex $v\in \Aast$,
        \item\label{it:3-mainlemma} $d(v,\Aast^{(1)}\Bast^{(k-2)})\geq (1-\rho) |\Aast| \binom{|\Bast|}{k-2}$ for any vertex $v\in \Bast$,
        \item\label{it:4-mainlemma} $d(L_0,\Aast^{(1)} \Bast^{(k - \l - 1)}), d(L_1,\Aast^{(1)}\Bast^{(k - \l - 1)}) \geq (1-\rho)|\Aast|\binom{|\Bast|}{k - \l - 1}$.
    \end{enumerate}
\end{lemma}

The next result, which we will use to conclude the proof of Theorem~\ref{theorem:extremal}, was obtained by Han and Zhao (see~\cite{HaZh15}*{Lemma~3.10}).

\begin{lemma}\label{lem:3.10}
    For any integers $k\geq 3$ and $1\leq \l < k/2$ there exists $\rho > 0$ such that the following holds.
    If $\cch$ is a sufficiently large $k$-uniform hypergraph with a partition $V(\cch)=\Aast\dcup \Bast$ and there exist two disjoint $\l$-sets $L_0,L_1\subset \Bast$ such that~\ref{it:1-mainlemma}--\ref{it:4-mainlemma} hold, then $\cch$ contains a Hamiltonian $\l$-path $\ccq'$ with $L_0$ and $L_1$ as ends.
    \qed
\end{lemma}

\section{Proof of the Main Lemma}\label{sec:mainproof}

We will start this section by describing the setup for the proof, which will be fixed for the rest of the paper.
Then we will prove some auxiliary lemmas and finally prove Lemma~\ref{lem:mainlemma}.
Let $\rho > 0$ and integers $k\geq 3$ and $1\leq \l<k/2$ be given.
Fix constants
\[
    \frac{1}{k}, \frac{1}{\l}, \rho \gg \delta \gg \eps \gg \epsprime \gg \theta \gg \xi.
\]
Let $n\in (k-\l)\bbn$ be sufficiently large and let $\cch$ be an $(\l,\xi)$-extremal $k$-uniform hypergraph on $n$ vertices that satisfies the $(k-2)$-degree condition
\begin{equation*}
    \delta_{k-2}(\cch)
    >
    \delta_{k-2}(\ccx_{k,\l}(n)).
\end{equation*}
Let $A \dcup B=V(\cch)$ be a minimal extremal partition of~$V(\cch)$, i.e.\ a partition satisfying
\begin{equation}
    \label{eq:partition-sizes}
    a = |A| = \left\lceil \frac{n}{2(k-\l)} \right\rceil - 1, \quad b = |B| = n - a, \qand e(B) \leq \xi \binom{n}{k},
\end{equation}
which minimises $e(B)$.
Recall that the extremal example $\ccx_{k,\l}(n)$ implies
\begin{equation}
    \label{eq:minimumDegree}
    \delta_{k-2}(\cch)
    >
    \binom{a}{2} + a(b - k + 2).
\end{equation}
Since $e(B)\leq \xi\binom{n}{k}$, we expect most vertices $v\in B$ to have low degree $d(v,B^{(k-1)})$ into~$B$.
Also, most $v\in A$ must have high degree $d(v,B^{(k-1)})$ into $B$ such that the degree condition for $(k-2)$-sets in~$B$ can be satisfied.
Thus, we define the sets $\Aeps$ and $\Beps$ to consist of vertices of high respectively low degree into~$B$ by
\begin{align*}
    \Aeps &= \left\lbrace v\in V\colon d(v,B^{(k-1)})\geq (1-\eps)\binom{|B|}{k-1}\right\rbrace,\\
    \Beps &=\left\lbrace v\in V\colon d(v,B^{(k-1)})\leq \eps\binom{|B|}{k-1}\right\rbrace,
\end{align*}
and set $\Veps=V\setminus (\Aeps \cup \Beps )$.
We will write $\aeps = |\Aeps|$, $\beps = |\Beps|$, and $\veps = |\Veps|$.
It follows from these definitions that
\begin{equation}\label{eq:Aeps-Beps-inclusion}
    \text{if } A\cap \Beps \neq \emptyset,
    \quad \text{then} \quad
    B \subset \Beps ,
    \quad \text{while otherwise} \quad
    A \subset \Aeps.
\end{equation}
For the first inclusion, consider a vertex $v \in A \cap \Beps$ and a vertex $w \in B \setminus \Beps$.
Exchanging~$v$ and~$w$ would create a minimal partition with fewer edges in $e(B)$, a contradiction to the minimality of the extremal partition.
The other inclusion is similarly implied by the minimality.

Actually, as we shall show below, the sets $\Aeps$ and $\Beps$ are not too different from $A$ and $B$ respectively:
\begin{equation}\label{eq:Aeps-Beps-sizes}
    |A\setminus \Aeps |, |B\setminus \Beps |, |\Aeps \setminus A|, |\Beps \setminus B|\leq \theta b \qand |\Veps|\leq 2\theta b.
\end{equation}
Note that by the minimum $(k-2)$-degree
\[
    \binom{a}{2}\binom{b}{k-2}+a\binom{b}{k-1}(k-1)
    <
    \binom{b}{k-2}\delta_{k-2}(\cch)
    \leq
    \sum_{S\in B^{(k-2)}} d(S).
\]
Every vertex $v \in |A\setminus \Aeps |$ satisfies $d(v, B^{(k-1)}) < (1 - \eps)\binom{b}{k-1}$, so we have
\begin{align*}
    \sum_{S\in B^{(k-2)}} d(S)
    \leq&
    \binom{a}{2}\binom{b}{k-2}+a\binom{b}{k-1}(k-1) \\
    &+ e(B)\binom{k}{2} - |A \setminus \Aeps| \eps \binom{b}{k-1}(k-1).
\end{align*}
Consequently $|A\setminus \Aeps |\leq \theta b$, as $e(B) < \xi \binom{n}{k}$ and $\xi \ll \theta, \eps$.

Moreover, $|B\setminus \Beps | \leq \theta b$ holds as a high number of vertices in $B \setminus \Beps $ would contradict $e(B) < \xi \binom{b}{k}$.
The other three inequalities~\eqref{eq:Aeps-Beps-sizes} follow from the already shown ones, for example for $|\Aeps \setminus A| < \theta b$ observe that
\[
    \Aeps \setminus A = \Aeps \cap B \subset B \setminus \Beps.
\]
Although the vertices in $\Beps$ were defined by their low degree into $B$, they also have low degree into the set $\Beps$ itself; for any $v \in \Beps$ we get
\begin{align*}
    d(v, \Beps^{(k - 1)})
    &\leq
    d(v, B^{(k - 1)}) + |\Beps\setminus B| \binom{|\Beps| - 1}{k - 2}\\
    &\leq
    \eps \binom{b}{k -1 } + \theta b |\Beps|^{k -1}\\
    &< 
    2 \eps \binom{|\Beps|}{k - 1}.
\end{align*}

Since we are interested in $\l$-paths, the degree of $\l$-tuples in $\Beps$ will be of interest, which motivates the following definition.
An $\l$-set $L\subset \Beps$ is called $\eps$-\emph{typical} if
\[
    d(L,B^{(k - \l)})\leq \eps\binom{|B|}{k-\l}.
\]
If $L$ is not $\eps$-typical, then it is called $\eps$-\emph{atypical}.
Indeed, most $\l$-sets in $\Beps$ are $\eps$-typical; denote by $x$ the number of $\eps$-atypical sets in~$\Beps$.
We have
\begin{equation}\label{eq:num-typical-sets}
    \frac{x\cdot\eps \binom{|B|}{k-\l}}{\binom{k}{\l}} \leq e(B \cup \Beps) \leq \xi \binom{n}{k} + \theta {|B|}^k,
    \quad \text{implying} \quad
    x\leq \epsprime \binom{|\Beps|}{\l}.
\end{equation}

\begin{lemma}\label{lem:typical-degree}
    The following holds for any $\Beps^{(m)}$-set $M$ if $m \leq k-2$.
    \[
        d(M,\Aeps^{(1)} \Beps^{(k - m - 1)}) + \frac{k-m}{2} d(M, \Beps^{(k - m)})
        \geq
        \left(1-\delta\right)|\Aeps|\binom{|\Beps| - m}{k-m-1}.
    \]
    In particular, the following holds for any $\eps$-typical $B^{(\l)}$-set $L$.
    \[
        d(L,\Aeps^{(1)} \Beps^{(k - \l - 1)})\geq (1-2\delta)|\Aeps|\binom{|\Beps|-\l}{k-\l-1}.
    \]
\end{lemma}

In the proof of the main lemma we will connect two $\eps$-typical sets only using vertices that are unused so far.
Even more, we want to connect two $\eps$-typical sets using exactly one vertex from $A$.
The following corollary of Lemma~\ref{lem:typical-degree} allows us to do this.

\begin{corollary}\label{corr:connect-extend-typical}
    Let $L$ and $L'$ be two disjoint $\eps$-typical sets in $\Beps$ and $U\subset V$ with~$|U| \leq \eps n$.
    Then the following holds.
    \begin{enumerate}[label=\alabel]
        \item\label{it:connect-typical} There exists an $\l$-path disjoint from $U$ of size two with ends $L$ and $L'$ that contains exactly one vertex from $ \Aeps$.
        \item\label{it:extend-typical} There exist $a \in \Aeps \setminus U$ and a set $(k - \l -1)$-set $C \subset \Beps \setminus U$ such that $L \cup a \cup C$ is an edge in $\cch$ and every $\l$-subset of $C$ is $\eps$-typical.
    \end{enumerate}
\end{corollary}
\begin{proof}[Proof of Corollary~\ref{corr:connect-extend-typical}]
    For~\ref{it:connect-typical}, the second part of Lemma~\ref{lem:typical-degree} for $L$ and $L'$ implies that they both extend to an edge with at least $(1 - 2\delta)|\Aeps|\binom{|\Beps| - \l}{k - \l - 1}$ sets in $\Aeps^{(1)}\Beps^{(k - \l - 1)}$.
    Only few of those intersect $U$ and by an averaging argument we obtain two sets $C, C' \in \Aeps^{(1)}\Beps^{(k - \l - 1)}$ such that $|C \cap C'| = \l$ and $L \cup C$ as well as $L' \cup C'$ are edges in $\cch$, which yields the required $\l$-path.
    In view of~\eqref{eq:num-typical-sets},~\ref{it:extend-typical} is a trivial consequence of the second part of Lemma~\ref{lem:typical-degree}.
\end{proof}
\begin{proof}[Proof of Lemma~\ref{lem:typical-degree}]
    Let $m \leq k-2$ and let $M \in \Beps^{(m)}$ be an $m$-set.
    We will make use of the following sum over all $(k-2)$-sets $D \subset \Beps$ that contain $M$.
    \begin{equation}\label{eq:sumdeg1}
        \begin{split}
        \sum_{\substack{M \subset D \subset \Beps\\|D| = k - 2}} d(D)
        =
        \sum_{\substack{M \subset D \subset \Beps \\ |D| = k - 2}} \Big(& d(D, \Aeps^{(1)} \Beps^{(1)}) + d(D, {(\Aeps \cup \Veps)}^{(2)}) \\
                                                                        & \qquad + d(D,\Beps^{(2)}) + d(D, \Veps^{(1)} \Beps^{(1)})\Big)
        \end{split}
    \end{equation}
    Note that we can relate the sums $\sum d(D, \Aeps^{(1)}\Beps^{(1)})$ and $\sum d(D,\Beps^{(2)})$ in~\eqref{eq:sumdeg1} to the terms in question as follows.
    \begin{equation}\label{eq:sumdeg2}
    \begin{split}
        d(M, \Aeps^{(1)} \Beps^{(k - m - 1)})
        &=
        \frac{1}{k - m - 1}\sum_{\substack{M \subset D \subset \Beps\\|D| = k - 2}} d(D, \Aeps^{(1)} \Beps^{(1)}),
        \\
        d(M, \Beps^{(k - m)})
        &=
        \frac{1}{\binom{k-m}{2}}\sum_{\substack{M \subset D \subset \Beps\\|D| = k - 2}} d(D, \Beps^{(2)}).
    \end{split}
    \end{equation}
    We will bound some of the terms on the right-hand side of~\eqref{eq:sumdeg1}.
    It directly follows from~\eqref{eq:Aeps-Beps-sizes} that $d(D, {(\Aeps\cup \Veps)}^{(2)})\leq \binom{a+3\theta b}{2}$; moreover, $d(D, \Veps^{(1)} \Beps^{(1)}) \leq 2 \theta b\beps$.
    Using the minimum $(k-2)$-degree condition~\eqref{eq:minimumDegree} we obtain
    \begin{equation*}
        \sum_{\substack{M \subset D \subset \Beps\\|D| = k - 2}} d(D)
        >
        \binom{\beps - m}{k - m - 2}\left(\binom{a}{2} + a(b - k + 2)\right).
    \end{equation*}
    Combining these estimates with~\eqref{eq:sumdeg1}~and~\eqref{eq:sumdeg2} yields
    \begin{align*}
        & d(M, \Aeps^{(1)} \Beps^{(k - m - 1)})
        + \frac{k-m}{2} d(M, \Beps^{(k - m)})
        \\
        & \quad \geq
        \frac{1}{k - m - 1}\binom{\beps - m}{k - m - 2}
        \left(\binom{a}{2} + a(b-k+2) - \binom{a+3\theta b}{2} - 2\theta b\beps \right)
        \\
        & \quad \geq
        \left(1-\delta\right)\aeps\binom{\beps - m}{k - m - 1}.
    \end{align*}
    For the second part of the lemma, note that the definition of $\eps$-typicality and $\eps \ll \delta$ imply that $\frac{k-\l}{2} d(L, \Beps^{(k - \l)})$ is smaller than $\delta \aeps \binom{\beps - \l}{k - \l - 1}$ for any $\eps$-typical $\l$-set $L$, which concludes the proof.
\end{proof}

For Lemma~\ref{lem:mainlemma}, we want to construct an $\l$-path $\ccq$, such that $\Veps \subset V(\ccq)$ and the remaining sets $\Aeps\setminus \ccq$ and $\Beps\setminus \ccq$ have the right relative proportion of vertices, i.e., their sizes are in a ratio of one to $(2k - 2\l - 1)$.
If $|A \cap \Beps| > 0$, then $B \subset \Beps$ (see~\eqref{eq:Aeps-Beps-inclusion}) and so $\ccq$ should cover $\Veps$ and contain the right number of vertices from $\Beps$.
For this, we have to find suitable edges inside $\Beps$, which the following lemma ensures.

\begin{lemma}\label{lem:2q-path}
    Suppose that $q = |A \cap \Beps| > 0$.
    Then there exist $2q + 2$ disjoint paths of size three, each of which contains exactly one vertex from $\Aeps$ and has two $\eps$-typical sets as its ends.
\end{lemma}
\begin{proof}
    We say that an $(\l - 1)$-set $M \subset \Beps$ is $\emph{good}$ if it is a subset of at least $(1 - \sqrt{\epsprime })\beps$ $\eps$-typical sets, otherwise we say that the set is \emph{bad}.
    We will first show that there are $2q+2$ edges in $\Beps$, each containing one $\eps$-typical and one \emph{good} $(\l - 1)$-set.
    Then we will connect pairs of these edges to $\l$-paths of size three.

    Suppose that $q = |A \cap \Beps| > 0$.
    So $B \subset \Beps$ by~\eqref{eq:Aeps-Beps-inclusion} and consequently $|\Beps| = |B| + q$ and $q \leq \theta |B|$.
    It is not hard to see from~\eqref{eq:num-typical-sets} that at most a $\sqrt{\epsprime}$ fraction of the $(\l - 1)$-sets in~$\Beps^{(l-1)}$ are bad.
    Hence, at least
    \[
        \left( 1 - \binom{k - 2}{\l}\epsprime - \binom{k - 2}{\l - 1}\sqrt{\epsprime} \right) \binom{b}{k-2}
    \]
    $(k -2)$-sets in $\Beps$ contain no $\eps$-atypical or bad subset.
    Let $\ccb \subset \Beps^{(k)}$ be the set of edges inside $\Beps$ that contain such a $(k-2)$-set.
    For all $M \in \Beps^{(k - 2)}$, by the minimum degree condition, we have $d(M, \Beps^{(2)}) \geq q(b-k+2) + \binom{q}{2}$ and, with the above, we have
    \begin{align*}
        |\ccb| & \geq \left(1 - \binom{k - 2}{\l}\epsprime - \binom{k - 2}{\l - 1}\sqrt{\epsprime }\right)\binom{b}{k-2}\frac{q(b-k+2)}{\binom{k}{2}} \nonumber \\
               & = \left(1 - \binom{k - 2}{\l}\epsprime - \binom{k - 2}{\l - 1}\sqrt{\epsprime }\right)\binom{b}{k-1}\frac{2q}{k} \geq \frac{q}{k} \binom{b}{k-1}.
    \end{align*}
    On the other hand, for any $v \in \Beps$ we have $d(v, \Beps^{(k - 1)}) < 2 \eps \binom{\beps}{k - 1}$ which implies that any edge in $\ccb$ intersects at most $2k \eps \binom{\beps}{k - 1}$ other edges in $\ccb$.
    So, in view of $\eps \ll \frac{1}{k}$ we may pick a set $\ccb'$ of $2q+2$ disjoint edges in $\ccb$.

    We will connect each of the edges in $\ccb'$ to an $\eps$-typical set.
    Assume we have picked the first $i-1$ desired $\l$-paths, say $\ccp_1, \dots, \ccp_{i-1}$, and denote by $U$ the set of vertices contained in one of the paths or one of the edges in $\ccb'$.
    For the rest of this proof, when we pick vertices and edges, they shall always be disjoint from $U$ and everything chosen before.
    Let $e$ be an edge in $\ccb'$ we have not considered yet and pick an arbitrary $\eps$-typical set $L' \subset \Beps \setminus U$.

    We will first handle the cases that $2\l + 1 < k$ or that $\l=1$, $k=3$.
    In the first case, a $(k-2)$-set that contains no $\eps$-atypical set already contains two disjoint $\eps$-typical sets.
    In the second case, an $\l$-set $\{v\}$ is $\eps$-typical for any vertex $v$ in $\Beps$ by the definition of $\eps$-typicality.
    Hence in both cases $e$ contains two disjoint $\eps$-typical sets, say $L_0$ and $L_1$.
    We can use Corollary~\ref{corr:connect-extend-typical}\,\ref{it:connect-typical}, as $|U| \leq 6kq$, to connect $L_1$ to $L'$ and obtain an $\l$-path $\ccp_i$ of size three that contains one vertex in $\Aeps$ and has $\eps$-typical ends $L_0$ and $L'$.

    So now assume that $2\l + 1 = k$ and $k > 3$, in particular $k - 2 = 2\l - 1$ and we may split the $(k-2)$-set considered in the definition of $\ccb$ into an $\eps$-typical $\l$-set $L$ and a good $(\l-1)$-set $G$.
    Moreover, let $w \in e \setminus (L \cup G)$ be one of the remaining two vertices and set~$N = G \cup {w}$.

    First assume that $d(N, \Aeps^{(1)} \Beps^{(\l)}) \geq \frac{\delta}{3} \aeps\binom{\beps}{\l}$.
    As $\theta \ll \delta$, at most $\frac{\delta}{3} \aeps \binom{b}{\l}$ sets in $\Aeps^{(1)} \Beps^{(\l)}$ intersect $U$.
    So it follows from Lemma~\ref{lem:typical-degree} that there exist $\Aeps^{(1)} \Beps^{(\l)}$-sets $C$, $C'$ such that $N \cup C$ and $L' \cup C'$ are edges, $|C \cap C'| = \l$ and~$|C \cap C' \cap \Aeps|=1$.

    Now assume that $d(N, \Aeps^{(1)} \Beps^{(\l)}) < \frac{\delta}{3} \aeps\binom{\beps}{\l}$.
    As the good set $G$ forms an $\eps$-typical set with most vertices in $\Beps$, there exists $v \in \Beps\setminus U$ such that
    \begin{equation*}
        d(N \cup \{v\}, \Aeps^{(1)} \Beps^{(\l - 1)}) < \delta \aeps\binom{\beps}{\l - 1}
    \end{equation*}
    and $G \cup \{ v \}$ is an $\eps$-typical set.
    Lemma~\ref{lem:typical-degree} implies that
    \begin{align*}
        d(N \cup \{v\}, \Beps^{(\l)})
        &\geq
        \frac{2}{\l} \left((1-\delta) \aeps \binom{\beps-(\ell+1)}{\l - 1} - \delta\aeps\binom{\beps}{\l - 1}\right)\\
        &\geq
        \frac{2}{\l} \left(\frac{1}{2} - 2\delta\right) \aeps \binom{\beps}{\l - 1}\\
        &\geq
        \delta \binom{\beps}{\l}.
    \end{align*}
    So there exists an $\eps$-typical $\l$-set $L^* \subset (\Beps \setminus U)$ such that $N \cup L^* \cup \{ v \}$ is an edge in $\cch$.
    Use Lemma~\ref{corr:connect-extend-typical}\,\ref{it:connect-typical} to connect $L^*$ to $L'$ and obtain an $\l$-path $\ccp_i$ of size three that contains one vertex in $\Aeps$ and has $\eps$-typical ends $G \cup \{ v \}$ and $L'$.
\end{proof}

If the hypergraph we consider is very close to the extremal example then Lemma~\ref{lem:2q-path} does not apply and we will need the following lemma.

\begin{lemma}\label{lem:one-or-two-edges}
    Suppose that $B = \Beps$.
    If $n$ is an odd multiple of $k-\ell$ then there exists a single edge on $\Beps$ containing two $\eps$-typical $\ell$-sets.
    If $n$ is an even multiple of $k-\ell$ then there either exist two disjoint edges on $\Beps$ each containing two $\eps$-typical $\ell$-sets or an $\ell$-path of size two with $\eps$-typical ends.
\end{lemma}
\begin{proof}
    For the proof of this lemma all vertices and edges we consider will always be completely contained in $\Beps$.
    First assume that there exists an $\eps$-atypical $\ell$-set $L$.
    Recall that this means that $d(L,B^{(k - \l)}) > \eps\binom{|B|}{k-\l}$ so in view of~\eqref{eq:num-typical-sets} and $\epsprime \ll \eps$ we can find two disjoint $(k-\l)$-sets extending it to an edge, each containing an $\eps$-typical set, which would prove the lemma.

    So we may assume that all $\ell$-sets in $\Beps^{(\l)}$ are $\eps$-typical.
    We infer from the minimum degree condition that $\Beps$ contains a single edge, which proves the lemma in the case that $n$ is an odd multiple of $k-\l$ and for the rest of the proof we assume that $n$ is an even multiple of $k-\l$.

    Assume for a moment that $\l = 1$.
    Recall that in this case any $(k-2)$-set in $B$ in the extremal hypegraph $\ccx_{k,\l}(n)$ is contained in one edge.
    Consequently, the minimum degree condition implies that any $(k-2)$-set in $\Beps$ extends to at least two edges on $\Beps$.
    Fix some edge $e$ in $\Beps$; any other edge on $\Beps$ has to intersect $e$ in at least two vertices or the lemma would hold.
    Consider any pair of disjoint $(k-2)$-sets $K$ and $M$ in $\Beps \setminus e$ to see that of the four edges they extend to, there is a pair which is either disjoint or intersect in one vertex, proving the lemma for the case $\l=1$.

    Now assume that $\l > 1$.
    In this case the minimum degree condition implies that any $(k-2)$-set in $\Beps$ extends to at least one edge on $\Beps$.
    Again, fix some edge $e$ in $\Beps$; any other edge on $\Beps$ has to intersect $e$ in at least one vertex or the lemma would hold.
    Applying the minimum degree condition to all $(k-2)$-sets disjoint from $e$ implies that one vertex $v \in e$ is contained in at least $\frac{1}{2k^2} \binom{|\Beps|}{k-2}$ edges.
    We now consider the $(k-1)$-uniform link hypergraph of $v$ on $\Beps$.
    Since any two edges intersecting in $\l-1$ vertices would finish the proof of the lemma, we may assume that there are no such pair of edges.
    However, a result of Frankl and Füredi~\cite[Theorem 2.2]{FranklFuredi} guarantees that this $(k-1)$-uniform hypergraph without an intersection of size $\l-1$ contains at most $\binom{|\Beps|}{k-\l-1}$ edges, a contradiction.
\end{proof}

The following lemma will allow us to handle the vertices in $\Veps$.

\begin{lemma}\label{lemma:pathV0}
    Let $U \subset \Beps$ with $|U| \leq 4k\theta$.
    There exists a family $\ccp_1, \ldots, \ccp_{\veps}$ of disjoint $\l$-paths of size two, each of which is disjoint from $U$ such that for all $i \in [\veps]$
    \[
        |V(\ccp_i) \cap \Veps| = 1 \qand |V(\ccp_i) \cap \Beps| = 2k - \l - 1,
    \]
    and both ends of $\ccp_i$ are $\eps$-typical sets.
\end{lemma}
\begin{proof}
    Let $\Veps = \{ x_1, \dots, x_{\veps} \}$.
    We will iteratively pick the paths.
    Assume we have already chosen $\l$-paths $\ccp_1, \dots, \ccp_{i-1}$ containing the vertices $v_1, \dots, v_{i-1}$ and satisfying the lemma.
    Let $U'$ be the set of all vertices in $U$ or in one of those $\l$-paths.
    From $v_i \notin \Beps$ we get
    \[
        d(v_i, \Beps^{(k-1)}) \geq d(v_i, B) - |B \setminus \Beps| \cdot \binom{|B|}{k-2} \geq \frac{\eps}{2} \binom{\beps}{k-1}.
    \]
    From~\eqref{eq:num-typical-sets} we get that at most $k^\l \epsprime \binom{\beps}{k-1}$ sets in $\Beps^{(k-1)}$ contain at least one $\eps$-atypical $\l$-set.
    Also, less than~$\frac{\eps}{8} \binom{\beps}{k-1}$ sets in $\Beps^{(k-1)}$ contain one of the vertices of $U'$.
    In total, at least $\frac{\eps}{4} \binom{\beps}{k-1}$ of the $\Beps^{(k-1)}$-sets form an edge with $v_i$.
    So we may pick two edges $e$ and $f$ in $\Veps^{(1)} \Beps^{(k-1)}$ that contain the vertex $v_i$ and intersect in $\l$ vertices.
    In particular, these edges form an $\l$-path of size two as required by the lemma.
\end{proof}

We can now proceed with the proof of Lemma~\ref{lem:mainlemma}.
Recall that we want to prove the existence of an $\l$-path $\ccq$ in $\cch$ with ends $L_0$ and $L_1$ and a partition
\[
    \Aast\dcup \Bast=(V\setminus \ccq)\dcup L_0\dcup L_1
\]
satisfying properties~\ref{it:1-mainlemma}--\ref{it:4-mainlemma} of Lemma~\ref{lem:mainlemma}.
Set $q = |A \cap \Beps|$.
We will split the construction of the $\l$-path $\ccq$ into two cases, depending on whether $q=0$ or not.

First, suppose that $q > 0$.
In the following, we denote by $U$ the set of vertices of all edges and $\l$-paths chosen so far.
Note that we will always have $|U| \leq 20 k \theta n$ and hence we will be in position to apply Corollary~\ref{corr:connect-extend-typical}.
We use Lemma~\ref{lem:2q-path} to obtain paths $\ccq_1, \ldots, \ccq_{2q+2}$ and then we apply Lemma~\ref{lemma:pathV0} to obtain $\l$-paths $\ccp_1, \ldots, \ccp_{\veps}$.
Every path $\ccq_i$, for $i \in [2q+2]$, contains $3k - 2\l - 1$ vertices from $\Beps$ and one from $\Aeps$, while every $\ccp_j$, for $j \in [\veps]$, contains $2k - \l - 1$ from $\Beps$ and one from $\Veps$.

As the ends of all these paths are $\eps$-typical, we apply Corollary~\ref{corr:connect-extend-typical}\,\ref{it:connect-typical} repeatedly to connect them to one $\l$-path $\ccp$.
In each of the $\veps + 2q + 1$ steps of connecting two $\l$-paths, we used one vertex from $\Aeps$ and $2k - 3\l - 1$ vertices from $\Beps$.
Overall, we have that
\[
    |V(\ccp) \cap \Aeps| = \veps + 4q + 3,
\]
as well as
\[
    |V(\ccp) \cap \Beps| = (4k - 4\l - 2)\veps + (5k - 5\l - 2)(2q+2) - (2k - 3\l - 1).
\]
Furthermore $|V(\ccp)| \leq 10 k \theta b$.

Using the identities $\aeps + \beps + \veps = n$ and $\aeps + q + \veps = a$, we will now establish property~\ref{it:1-mainlemma} of Lemma~\ref{lem:mainlemma}.
Set $s(\ccp) = (2k - 2\l - 1)|\Aeps\setminus V(\ccp)| - |\Beps\setminus V(\ccp)| - 2\l$, so
\begin{align*}
    s(\ccp) & = (2k - 2\l - 1)|\Aeps \setminus V(\ccp)| - |\Beps \setminus V(\ccp)| - 2\l\\
            & = (2k - 2\l - 1)(\aeps - (\veps + 4q + 3)) - \beps \\
            & \phantom{{} = {}} {} + (4k - 4\l -2)\veps + (5k - 5\l - 2)(2q+2) - (2k - 3\l - 1) - 2\l \\
            & = (2k - 2\l - 1)\aeps - \beps + (2k - 2\l - 1)\veps + 2(k - \l)q + 2k - 3\l \\
            & = 2(k - \l)(\aeps + \veps + q + 1) - n - \l \\
            & = 2(k - \l)(a + 1) - n - \l.
\end{align*}
If $n/(k - \l)$ is even, $s(\ccp) = -\l$ (see~\eqref{eq:partition-sizes}) and we set $\ccq = \ccp$.
Otherwise $s(\ccp) = k - 2\l$ and we use Corollary~\ref{corr:connect-extend-typical}\,\ref{it:extend-typical} to append one edge to $\ccp$ to obtain $\ccq$.
It is easy to see that one application of Corollary~\ref{corr:connect-extend-typical}\,\ref{it:extend-typical} decreases $s(\ccp)$ by $k - \l$.
Setting $\Aast = \Aeps\setminus V(\ccq)$ and $\Bast = (\Beps \setminus V(\ccq)) \cup L_0 \cup L_1$ we get from $s(\ccq) = -\l$ that $\Aast$ and $\Bast$ satisfy~\ref{it:1-mainlemma}.

Now, suppose that $q = 0$.
Apply Lemma~\ref{lemma:pathV0} to obtain $\l$-paths $\ccp_1, \ldots, \ccp_{\veps}$.
If $B=\Beps$, apply Lemma~\ref{lem:one-or-two-edges} to obtain one or two more $\l$-paths contained in $\Beps$.
We apply Corollary~\ref{corr:connect-extend-typical}\,\ref{it:connect-typical} repeatedly to connect them to one $\l$-path $\ccp$.

Since $q = 0$, we have that $\Beps \subset B$ and $\aeps + \veps = |V\setminus \Beps| = a + |B\setminus \Beps|$.
We can assume without loss of generality that $\Veps \neq \emptyset$, otherwise just take $\Veps = \{ v \}$ for an arbitrary $v \in V (\cch)$.
If $B=\Beps$ let $x$ be $2(k-\l)$ or $k-\l$ depending on whether $n$ is an odd or even multiple of $k-\l$; otherwise let $x=0$.
With similar calculations as before and the same definition of $s(\ccp)$ we get that
\[
    s(\ccp) = 2(k-\l)a  + x +  2(k-\l)|B \setminus \Beps| - n - \l  \equiv -\l \mod(k - \l).
\]
Extend the $\l$-path $\ccp$ to an $\l$-path $\ccq$ by adding $\frac{s(\ccp) + \l}{k - l}$ edges using Corollary~\ref{corr:connect-extend-typical}\,\ref{it:extend-typical}.
Thus $s(\ccq) = -\l$, and we get~\ref{it:1-mainlemma} as in the previous case.

In both cases, we will now use the properties of the constructed $\l$-path $\ccq$ to show~\ref{it:2-mainlemma}-\ref{it:4-mainlemma}.
We will use that $v(\ccq) \leq 20 k \theta b$, which follows from the construction.
Since $\Aast \subset \Aeps$, for all $v \in \Aast$ we have $d(v, B^{(k-1)}) \geq (1 - \eps)B^{(k-1)}$.
Thus
\[
    d(v, \Bast^{(k-1)}) \geq d(v, B^{(k-1)}) - |\Bast\setminus B|\binom{|\Bast| - 1}{k-2} \geq (1 - 2\eps)\binom{|\Bast|}{k-1},
\]
which shows~\ref{it:2-mainlemma}.

For~\ref{it:3-mainlemma}, Lemma~\ref{lem:typical-degree} yields for all vertices $v \in \Bast \subset \Beps$ that
\[
    d(v,\Aeps^{(1)} \Beps^{(k-2)}) + \frac{k-1}{2} d(v, \Beps^{(k-1)})
    \geq
    \left(1-\delta\right)|\Aeps|\binom{|\Beps| - 1}{k-2}.
\]
The second term on the left can be bounded from above by $2k\eps\binom{\beps}{k-1}$.
So, as $\delta, \eps \ll \rho$ and $\aeps - |\Aast| \ll \rho |\Aast| $ as well as $\beps - |\Bast| \ll \rho |\Bast|$, we can conclude~\ref{it:3-mainlemma}.

By Lemma~\ref{lem:typical-degree}, we know that
\[
    d(L_{0}, \Aeps^{(1)}\Beps^{(k - 1)}),d(L_{1}, \Aeps^{(1)}\Beps^{(k - 1)}) \geq (1 - \delta) \aeps\binom{\beps - \l}{k - \l - 1}.
\]
As $\delta \ll \rho$ and $\aeps - |\Aast| \ll \rho |\Aast| $ as well as $\beps - |\Bast| \ll \rho |\Bast|$, we can conclude~\ref{it:4-mainlemma}.

\end{document}